\documentclass[12pt]{article}
\usepackage{amsmath, amsthm}\usepackage{enumerate}\usepackage{amssymb}
\usepackage{bbold}
\usepackage{bbm}\usepackage{dsfont}\usepackage{color}\usepackage{xcolor}
\usepackage[explicit]{titlesec}
\newtheorem{theorem}{Theorem}
\newtheorem{proposition}[theorem]{Proposition}

\newtheorem{lemma}[theorem]{Lemma}
\newtheorem{example}[theorem]{Example}
\newtheorem{corollary}[theorem]{Corollary}

\newtheorem{assertion}[theorem]{Assertion}

\newtheorem{definition}[theorem]{Definition}

\makeatletter
\renewcommand{\subsection}{\@startsection{subsection}{1}
{0pt}{3.25ex plus 1ex minus.2ex}{-1em}{\normalfont\normalsize\bf}}
\makeatother
\begin{document}

\title{{\bf Enveloping norms on the spaces of regularly 
$\cal{P}$-operators in Banach lattices}}
\maketitle
\author{\centering{{Safak Alpay$^{1}$, Eduard Emelyanov$^{2}$, Svetlana Gorokhova $^{3}$\\ 
\small $1$ Department of Mathematics, Middle East Technical University, 06800 Ankara, Turkey, safak@metu.edu.tr\\ 
\small $2$ Sobolev Institute of Mathematics, Acad. Koptyug avenue, 4, 630090, Novosibirsk, Russia, emelanov@math.nsc.ru\\ 
\small $3$ Uznyj matematiceskij institut VNC RAN, Vatutin str., 53, 362025, Vladikavkaz, Russia, lanagor71@gmail.com}
\abstract{We introduce and study the enveloping norms of regularly 
$\cal{P}$-operators between Banach lattices 
$E$ and $F$, where $\cal{P}$ is a subspace of the space
$\text{\rm L}(E,F)$ of continuous operators from $E$ to $F$.
We prove that if $\cal{P}$ is closed in $\text{\rm L}(E,F)$
in the operator norm then the regularly $\cal{P}$-operators forms a Banach space under
the enveloping norm. Conditions providing that regularly $\cal{P}$-operators forms 
a Banach lattice under the enveloping norm are given.}
\vspace{5mm}


}}
\bigskip
\bigskip

Throughout the paper, all vector spaces are real, all operators are linear, 
letters $E$ and $F$ denote Banach latices, 
$B_E$ denotes the closed unit ball of $X$,
$\text{\rm L}(E,F)$ (\text{\rm K}(E,F); \text{\rm W}(E,F)) denotes
the space of all continuous 
(resp. compact; weakly compact) operators 
from $E$ to $F$, and $E_+$ the positive cone of $E$. 
Let ${\cal P}\subseteq\text{\rm L}(E,F)$. We call elements of ${\cal P}$ by
${\cal P}$-operators and say that ${\cal P}$-operators satisfy the 
{\em domination property} if $0\le S\le T\in{\cal P}$ implies
$S\in{\cal P}$. An operator $T\in\text{\rm L}(E,F)$ is called
${\cal P}$-{\em dominated} if $\pm T\le U$ for some  $U\in{\cal P}$.
Under the assumption ${\cal P}\pm{\cal P}\subseteq{\cal P}\ne\emptyset$,
${\cal P}$-operators satisfy the domination property iff each 
${\cal P}$-dominated operator lies in ${\cal P}$. 
We refer to \cite{AlBu,GM,Mey} for unexplained terminology and notations.

\subsection{}
In the present paper we continue the investigation of regularly ${\cal P}$-operators 
following \cite{Emel,AEG3} and introduce enveloping norms on spaces 
of regularly ${\cal P}$-operators.
An operator $T:E\to F$ is called {\em regular} 
if $T=T_1-T_2$ for some $T_1,T_2\in\text{\rm L}_+(E,F)$. 
We denote by $\text{\rm L}_r(E,F)$ 
(resp. $\text{\rm L}_{ob}(E,F)$, $\text{\rm L}_{oc}(E,F)$) the ordered space of all regular 
(resp. order bounded, \text{\rm o}-continuous) operators 
in $\text{\rm L}(E,F)$.  The space $\text{\rm L}_r(E,F)$ is a Banach space under 
the {\em regular norm} \cite[Prop.1.3.6]{Mey}
\begin{equation}\label{def of regular norm}
   \|T\|_r:=\inf\{\|S\|:\pm T\le S\in\text{\rm L}(E,F)\}.
\end{equation}
Furthermore, for every $T\in\text{\rm L}_r(E,F)$,
\begin{equation}\label{regular norm 1}
   \|T\|_r=\inf\{\|S\|: S\in\text{\rm L}(E,F), |Tx|\le S|x|\ \forall x\in E\}\ge\|T\|,
\end{equation} 
and if $F$ is Dedekind complete then $(\text{\rm L}_r(E,F),\|\cdot\|_r)$ is a Banach lattice
with $\|T\|_r=\|~|T|~\|$ for each $T\in\text{\rm L}_r(E,F)$ \cite[Prop.1.3.6]{Mey}.
The following notion was introduced in \cite[Def.2]{Emel} 
(cf. also \cite[Def.1.5.1]{AEG3}). 

\begin{definition}\label{rP-operators}{\em
Let ${\cal P}\subseteq\text{\rm L}(E,F)$. An operator $T:E\to F$ is called 
a {\it regularly} ${\cal P}$-{\it operator}
(shortly, an \text{\rm r}-${\cal P}$-{\it operator}) if $T=T_1-T_2$ 
with $T_1,T_2\in{\cal P}\cap\text{\rm L}_+(E,F)$.
We denote by: 
\begin{enumerate}[]
\item 
${\cal P}(E,F):= {\cal P}$ the set of all ${\cal P}$-{\em operators} 
in $\text{\rm L}(E,F)$;  
\item 
${\cal P}_r(E,F)$ the set of all regular operators in ${\cal P}(E,F)$;  
\item 
$\text{\rm r-}{\cal P}(E,F)$ the set of 
all \text{\rm r}-${\cal P}$-operators in $\text{\rm L}(E,F)$.
\end{enumerate}
}
\end{definition}

\begin{assertion}\label{prop elem}
{\rm (\cite[Prop.1.5.2]{AEG3})}
Let ${\cal P}\subseteq\text{\rm L}(E,F)$, ${\cal P}\pm{\cal P}\subseteq{\cal P}\ne\emptyset$, 
and $T\in\text{\rm L}(E,F)$. Then the following holds.
\begin{enumerate}[\em (i)]
\item 
$T$ is an {\em r-}${\cal P}$-operator iff $T$ is a ${\cal P}$-dominated ${\cal P}$-operator.  
\item 
Suppose ${\cal P}$-operators satisfy the domination property and the modulus $|T|$ exists 
in $\text{\rm L}(E,F)$. Then $T$ is an \text{\rm r}-${\cal P}$-operator iff  $|T|\in{\cal P}$.
\end{enumerate}
\end{assertion}
\noindent
The next fact was established in \cite[Prop.1.5.3]{AEG3}.

\begin{assertion}\label{vect lat}
$\text{\rm r-}{\cal P}(E,F)$ is a Dedekind complete vector lattice if
$F$ is Dedekind complete and ${\cal P}$ is a subspace of $\text{\rm L}(E,F)$
satisfying the domination property. 
\end{assertion}

\subsection{}
In the definition (\ref{def of regular norm}) of the regular norm,
we replace $\text{\rm L}(E,F)$ by an arbitrary subspace ${\cal P}$ 
of $\text{\rm L}(E,F)$ as follows:
\begin{equation}\label{enveloping norm}
   \|T\|_{\text{\rm r-}{\cal P}}:=\inf\{\|S\|:\pm T\le S\in{\cal P}\} \ \ 
   \ \ (T\in\text{\rm r-}{\cal P}(E,F)).
\end{equation}

\begin{lemma}\label{Enveloping P-norm}
For a subspace ${\cal P}$ of $\text{\rm L}(E,F)$, the formula 
$(\ref{enveloping norm})$ defines a norm on $\text{\rm r-}{\cal P}(E,F)$,
called the {\rm enveloping norm}. Moreover,
\begin{equation}\label{Enveloping P-norm 2}
   \|T\|_{\text{\rm r-}{\cal P}}=
   \inf\{\|S\|: S\in{\cal P}\ \&\ (\forall x\in E)\ |Tx|\le S|x|\}
   \ \ \ (T\in\text{\rm r-}{\cal P}(E,F)).
\end{equation}
If ${\cal P}_1$ is a subspace of ${\cal P}$ then
\begin{equation}\label{Enveloping P-norm 1}
   \|T\|_{\text{\rm r-}{\cal P}_1}\ge\|T\|_{\text{\rm r-}{\cal P}}\ge\|T\|_r\ge\|T\| 
   \ \ \ \ \ (\forall\ T\in\text{\rm r-}{\cal P}_1(E,F)).
\end{equation}
\end{lemma}

\begin{proof}
Only the triangle inequality and the 
formula (\ref{Enveloping P-norm 2}) require explanations.

(A) Let $T_1,T_2\in\text{\rm r-}{\cal P}(E,F)$ and $\varepsilon>0$. Pick
$S_1,S_2\in{\cal P}$ with $\pm T_1\le S_1$, $\pm T_2\le S_2$,
$\|S_1\|\le\|T_1\|_{\text{\rm r-}{\cal P}}+\varepsilon$, and 
$\|S_2\|\le\|T_2\|_{\text{\rm r-}{\cal P}}+\varepsilon$.
Then $\pm(T_1+T_2)\le S_1+S_2\in{\cal P}$, and
$
    \|T_1+T_2\|_{\text{\rm r-}{\cal P}}\le\|S_1+S_2\|\le\|S_1\|+\|S_2\|\le
    \|T_1\|_{\text{\rm r-}{\cal P}}+\|T_2\|_{\text{\rm r-}{\cal P}}+2\varepsilon.
$
Since $\varepsilon>0$ is arbitrary, 
$\|T_1+T_2\|_{\text{\rm r-}{\cal P}}\le
\|T_1\|_{\text{\rm r-}{\cal P}}+\|T_2\|_{\text{\rm r-}{\cal P}}$.

(B) Denote the right side of (\ref{Enveloping P-norm 2}) by $R(T)$. 
If $\pm T\le S\in{\cal P}$ then 
$$
   \pm Tx=\pm(T(x_+)-T(x_-))=\pm T(x_+)\mp T(x_-)\le S(x_+)+S(x_-)=S|x|
$$
for all $x\in E$. Then $|Tx|\le S|x|$ for all $x\in E$, and hence 
$\|T\|_{\text{\rm r-}{\cal P}}\ge R(T)$.

If $S\in{\cal P}$ satisfies $|Tx|\le S|x|$ for all $x\in E$ then,
for all $y\in E_+$, $|Ty|\le Sy$ and consequently $\pm Ty\le Sy$.
Therefore $\pm T\le S$, and hence $R(T)\ge\|T\|_{\text{\rm r-}{\cal P}}$.
\end{proof}
\noindent
Let ${\cal P}$ be a subspace of $\text{\rm L}(E,F)$. 
What are the conditions for ${\cal P}_r(E,F)=\text{\rm r-}{\cal P}(E,F)$? Trivially it happens
for ${\cal P}=\text{\rm L}(E,F)$. The following proposition shows that if $\text{\rm L}(E,F)$
is non-trival then we can always find a closed subspace ${\cal P}$ such that 
$\text{\rm r-}{\cal P}(E,F)$ is a proper subspace of ${\cal P}_r(E,F)$.
More interesting case of the question is included in Example~\ref{Krengel}
below.

\begin{proposition}
Let $\dim(\text{\rm L}(E,F))>1$. Then there exists a one-dimensional subspace 
${\cal P}$ of $\text{\rm L}(E,F)$
such that $\text{\rm r-}{\cal P}(E,F)\subsetneqq{\cal P}_r(E,F)$.
\end{proposition}

\begin{proof}
If $\dim(E)>1$ then $\dim(E')>1$ and hence 
there are two linearly independent functionals
$g_1,g_2\in E'_+$. We take a nonzero
element $y\in F_+$ and set ${\cal P}:=\text{\rm span}((g_1-g_2)\otimes y)$.

If $\dim(F)>1$, there are two linearly independent elements
$y_1,y_2\in F_+$. We take a nonzero
functional  $g\in E'_+$ and set ${\cal P}:=\text{\rm span}(g\otimes (y_1-y_2))$.

In the both cases ${\cal P}$ is one-dimensional
subspace of $\text{\rm L}(E,F)$.
Moreover, $\text{\rm r-}{\cal P}(E,F)=\{0\}\ne {\cal P}={\cal P}_r(E,F)$.
\end{proof}

\subsection{}
The following result is an extension of \cite[Prop.1.3.6]{Mey}, 
\cite[Lm.1]{Emel} (see also \cite[Prop.2.2]{CW97} and 
\cite[Thm.2.3]{Cheng} for particular cases). Its proof is a modification 
of the proof of \cite[Prop.1.3.6]{Mey}.

\begin{theorem}\label{P-norm}
Let ${\cal P}$ be a subspace of $\text{\rm L}(E,F)$
closed in the operator norm.
Then $\text{\rm r-}{\cal P}(E,F)$ is a Banach space under the
enveloping norm.
\end{theorem}

\begin{proof}
Let $(T_n)$ be a
$\|\cdot\|_{\text{\rm r-}{\cal P}}$-Cauchy sequence in $\text{\rm r-}{\cal P}(E,F)$,
say $T_n=P_n-R_n$ for $P_n,R_n\in{\cal P}\cap{\cal L}_+(E,F)$.
Without lost of generality, we can assume 
$$
   \|T_{n+1}-T_n\|_{\text{\rm r-}{\cal P}}<2^{-n} \ \ \ (\forall n\in\mathbb{N}). 
$$
Since $\|\cdot\|_{\text{\rm r-}{\cal P}}\ge\|\cdot\|$, there exists a
$T\in\text{\rm L}(E,F)$ with $\|T-T_n\|\to 0$.
We obtain $T\in{\cal P}$ because $T_n\in{\cal P}$ and ${\cal P}$ is closed 
in the operator norm. Pick $S_n\in{\cal P}$ with $\|S_n\|<2^{-n}$ and
$\pm(T_{n+1}-T_n)\le S_n$. By (\ref{Enveloping P-norm 2}), 
$$
   |(T_{n+1}-T_n)x|\le S_n|x| \ \ \ \ \ (\forall x\in E)(\forall n\in\mathbb{N}). 
$$
As ${\cal P}$ is closed,
$Q_n:=\|\cdot\|\text{\rm -}\sum\limits_{k=n}^\infty S_k\in{\cal P}$ 
for each $n\in\mathbb{N}$.
Since
$$
   |(T-T_n)x|=\lim\limits_{k\to\infty}|(T_k-T_n)x|\le
   \sum\limits_{k=n}^\infty|(T_{k+1}-T_n)x|\le Q_n|x| \ \ \ \ \ (\forall x\in E), 
$$
then $\pm(T-T_n)\le Q_n$ for all $n\in\mathbb{N}$. 
Thus $-Q_n\le(T-T_n)\le Q_n$ and hence $0\le(T-T_n)+Q_n$
for all $n\in\mathbb{N}$. In particular, 
$$
   T=[(T-T_n)+Q_n]+[T_n-Q_n]=
   [(T-T_n)+Q_n+P_n]-[R_n+Q_n]\in\text{\rm r-}{\cal P}(E,F),
$$
and hence $(T-T_n)\in\text{\rm r-}{\cal P}(E,F)$ for all $n\in\mathbb{N}$.
Now, 
$$
   \|T-T_n\|_{\text{\rm r-}{\cal P}}\le\|Q_n\|<2^{1-n}
$$
implies $(T_n)\stackrel{\|\cdot\|_{\text{\rm r-}{\cal P}}}{\to}T$.
\end{proof}

\begin{theorem}\label{reg op banach lattice}
Let $F$ be Dedekind complete and let ${\cal P}$ 
be a closed in the operator norm subspace of $\text{\rm L}(E,F)$
satisfying the domination property. Then $\text{\rm r-}{\cal P}(E,F)$
is a Dedekind complete Banach lattice under the enveloping norm.
\end{theorem}

\begin{proof}
By Assertion~\ref{vect lat}, $\text{\rm r-}{\cal P}(E,F)$ is a Dedekind 
complete vector lattice. Theorem~\ref{P-norm} implies that 
$\text{\rm r-}{\cal P}(E,F)$ is $\|\cdot\|_{\text{\rm r-}{\cal P}}$-complete.
Let $|S|\le|T|$ for some $S,T\in\text{\rm r-}{\cal P}(E,F)$. 
By Assertion \ref{prop elem}~(ii), $|S|,|T|\in{\cal P}$, and hence
$$
   \|S\|_{\text{\rm r-}{\cal P}}=\|~|S|~\|=
   \sup\limits_{x\in E_+\cap B_E}\||S|x\|\le
   \sup\limits_{x\in E_+\cap B_E}\||T|x\|=\|~|T|~\|=\|T\|_{\text{\rm r-}{\cal P}}.
$$
Therefore,  $(\text{\rm r-}{\cal P}(E,F),\|\cdot\|_{\text{\rm r-}{\cal P}})$ 
is a Banach lattice.
\end{proof}

\subsection{}
In general, $\text{\rm r-}{\cal P}(E,F)\subsetneqq\text{\rm r-}{\overline{\cal P}}(E,F)$,
where ${\overline{\cal P}}$ is the norm-closure of ${\cal P}$ in $\text{\rm L}(E,F)$.
From the other hand, for ${\cal P}:=\text{\rm r-L}_{ob}(E,F)$ (which  
is almost never closed in $\text{\rm L}(E,F)$ in the operator norm),
$\text{\rm r-}{\cal P}(E,F)=\text{\rm r-}{\overline{\cal P}}(E,F)=\text{\rm L}_r(E,F)$.
The enveloping norms on these three spaces agree with the regular norm.
The next result coupled with Example~\ref{P=o-cont} shows that the enveloping norm 
on $\text{\rm r-}{\cal P}(E,F)$ can be 
complete even if ${\cal P}(E,F)$ is not complete in the operator norm.

\begin{proposition}\label{o-cont not so bad}
Let the norm in $F$ be \text{\rm o}-continuous. Then
$\text{\rm r-L}_{oc}(E,F)$ is a Banach space 
under the enveloping norm.
\end{proposition}

\begin{proof}
Let $(T_n)$ be a Cauchy sequence in $\text{\rm r-L}_{oc}(E,F)$ in 
the enveloping norm. We can assume that 
$\|T_{n+1}-T_n\|_{\text{\rm r-L}_{oc}(E,F)}<2^{-n}$ for all
$n\in\mathbb{N}$. Let $T\in\text{\rm L}(E,F)$ satisfy
$\|T-T_n\|\to 0$. Pick $S_n\in\text{\rm L}_{oc}(E,F)$ with $\|S_n\|<2^{-n}$ and
$\pm(T_{n+1}-T_n)\le S_n$. First, we claim
$Q_n:=\|\cdot\|\text{\rm -}\sum\limits_{k=n}^\infty S_k\in\text{\rm L}_{oc}(E,F)$ 
for all $n\in\mathbb{N}$. To prove the claim, it suffices to show that
$Q_1\in\text{\rm L}_{oc}(E,F)$. So, let $x_\alpha\downarrow 0$ in $E$.
Passing to a tail we can assume that $\|x_\alpha\|\le M\in\mathbb{R}$
for all $\alpha$. Since $Q_1\ge 0$ then $Q_1x_\alpha\downarrow\ge 0$ and hence
in order to show that $Q_1x_\alpha\downarrow 0$ it is enough to
prove that $\|Q_1x_\alpha\|\to 0$.
Let $\varepsilon>0$. Fix an $m\in\mathbb{N}$ with
$\|Q_{m+1}\|\le\varepsilon$. Since the positive 
operators $S_1,...,S_m$ are all \text{\rm o}-continuous, 
and since the norm in $F$ is \text{\rm o}-continuous,
there exists an $\alpha_1$ such that 
$\sum\limits_{k=1}^m\|S_kx_\alpha\|\le\varepsilon$
for all $\alpha\ge\alpha_1$.  Since $\varepsilon>0$ is arbitrary,
it follows from 
$$
   \|Q_1x_\alpha\|\le\|\sum\limits_{k=1}^mS_kx_\alpha\|+
   \|Q_{m+1}x_\alpha\|\le\varepsilon+M\|Q_{m+1}\|\le 
   2\varepsilon \ \ \ \ (\forall\alpha\ge\alpha_1)
$$
that $\|Q_1x_\alpha\|\to 0$, which proves our claim that 
$Q_n\in\text{\rm L}_{oc}(E,F)$
for all $n\in\mathbb{N}$.

Since $\pm(T_{n+1}-T_n)\le S_n$ then
by formula (\ref{Enveloping P-norm 2}),
$$
   |(T-T_n)x|=\lim\limits_{k\to\infty}|(T_k-T_n)x|\le
   \sum\limits_{k=n}^\infty |(T_{k+1}-T_n)x|\le
   \sum\limits_{k=n}^\infty S_n|x|=Q_n|x|  
$$
for all $x\in E$. In particular, $|T-T_1|\le Q_1\in\text{\rm L}_{oc}(E,F)$ and,
since $\text{\rm L}_{oc}(E,F)$ is an order ideal in $\text{\rm L}_{r}(E,F)$, then
$(T-T_1)\in\text{\rm L}_{oc}(E,F)$. 
Since $T_1\in\text{\rm L}_{oc}(E,F)$, it follows
$T\in\text{\rm L}_{oc}(E,F)$.
Now, $\|T-T_n\|_{\text{\rm r-L}_{oc}(E,F)}\le\|Q_n\|<2^{1-n}$ implies
$(T_n)\stackrel{\|\cdot\|_{\text{\rm r-L}_{oc}(E,F)}}{\longrightarrow}T$.
\end{proof}

\begin{example}\label{P=o-cont}
Consider the following modification of Krengel's example {\rm \cite{Kre} (}cf. 
{\rm \cite[Ex.5.6]{AlBu}}$)$.
Define a sequence $(S_n)$ in $\text{\rm L}_{oc}((\oplus_{n=1}^{\infty}\ell^2_{2^n})_0)$ by
$$
   S_nx:=(2^{-\frac{1}{3}}T_1x_1,2^{-\frac{2}{3}}T_2x_2,\dots 2^{-\frac{n}{3}}T_nx_n,0,0,\dots),
$$
where $T_n:\ell^2_{2^n}\to\ell^2_{2^n}$ is an isometry as in {\rm \cite[Ex.5.6]{AlBu}}.
Then $\|S_n-S\|\to 0$ for an operator $S\in\text{\rm L}(E,F)$, defined by
$$
   Sx:=(2^{-\frac{1}{3}}T_1x_1,2^{-\frac{2}{3}}T_2x_2,\dots2^{-\frac{n}{3}}T_nx_n,\dots).
$$
Notice that $(\oplus_{n=1}^{\infty}\ell^2_{2^n})_0$ has \text{\rm o}-continuous norm,
and hence is Dedekind complete.
However $|S|$ does not exist. Thus, $S$ is not order bounded and hence
$S\not\in\text{\rm L}_{oc}((\oplus_{n=1}^{\infty}\ell^2_{2^n})_0)$.
\end{example}

\subsection{}
We include several applications of Theorem \ref{P-norm} and Assertion \ref{prop elem}.
We begin with recalling the following definition (see \cite{BD,El}).

\begin{definition}\label{Main def of limited operators} 
{\em A continuous operator $T:E\to F$ is called
{\em limited} (resp. {\em almost limited}) if $T'$ takes $\text{\rm w}^\ast$-null 
(resp. disjoint $\text{\rm w}^\ast$-null) sequences of $F'$ to norm null sequences of $E'$.
By $\text{\rm Lm}(E,F)$ (by $\text{\rm a-Lm}(E,F)$) we denote the space 
of limited (resp. almost limited) operators from $E$ to $F$.}
\end{definition}

\begin{lemma}\label{alimited are closed}
The spaces $\text{\rm a-Lm}(E,F)$ and $\text{\rm Lm}(E,F)$
are both closed in $\text{\rm L}(E,F)$ under the operator norm.
\end{lemma}

\begin{proof}
Suppose $\text{\rm a-Lm}(E,F)\ni T_k\stackrel{\|\cdot\|}{\to}T\in\text{\rm L}(E,F)$.
Let $(f_n)$ be disjoint $\text{\rm w}^\ast$-null in $F'$.
We need to show that $(T'f_n)$ is norm null $X'$. Let $\varepsilon>0$.
Choose any $k$ with $\|T'_k-T'\|\le\varepsilon$. Since
$T_k\in\text{\rm a-Lm}(X,F)$, there exists $n_0$ such that
$\|T_k'f_n\|\le\varepsilon$ whenever $n\ge n_0$.
As $(f_n)$ is $\text{\rm w}^\ast$-null, there exists $M\in\mathbb{R}$
such that $\|f_n\|\le M$ for all $n\in\mathbb{N}$.
Then 
$$
   \|T'f_n\|\le\|T'_kf_n\|+\|T'_kf_n-T'f_n\|\le
   \varepsilon+\|T'_k-T'\|\|f_n\|\le\varepsilon+M\varepsilon
$$
for $n\ge n_0$. Since $\varepsilon>0$ is arbitrary
then $(T'f_n)$ is norm null, as desired.

Similarly, the space $\text{\rm Lm}(E,F)$ is closed.
\end{proof}
\noindent
The next result follows from Theorem \ref{P-norm} 
and Lemma \ref{alimited are closed}.

\begin{corollary}\label{P-norm in r-a-Lm}
For arbitrary Banach lattices $E$ and $F$, 
$\text{\rm r-Lm}(E,F)$ and $\text{\rm r-a-Lm}(E,F)$ 
are both Banach spaces {\rm (}each under its enveloping norm{\rm )}.
\end{corollary}

\begin{definition}\label{Main Grothendieck operators} {\rm (see \cite{GM})}
{\em A continuous operator $T:E\to F$ is called 
{\em Grothendieck} (resp. {\em almost Grothendieck})
if $T'$ takes $\text{\rm w}^\ast$-null 
(resp. disjoint $\text{\rm w}^\ast$-null)
sequences of $F'$ to \text{\rm w}-null sequences of $E'$.
By $\text{\rm G}(E,F)$ (by $\text{\rm a-G}(E,F)$) we denote the space 
of Grothendieck (resp. almost Grothendieck) operators from $E$ to $F$.
}
\end{definition}
\noindent
By Definitions \ref{Main def of limited operators} and \ref{Main Grothendieck operators},  
\begin{equation}\label{limgr1}
   \text{\rm K}(E,F)\subseteq\text{\rm Lm}(E,F)\subseteq
   \text{\rm a-Lm}(E,F)\subseteq\text{\rm a-G}(E,F),  
\end{equation}
and
\begin{equation}\label{limgr2}
    \text{\rm Lm}(E,F)\subseteq\text{\rm G}(E,F)\subseteq\text{\rm a-G}(E,F). 
\end{equation}

\begin{lemma}\label{a-G are closed}
The spaces $\text{\rm a-G}(E,F)$ and $\text{\rm G}(E,F)$
are both closed in $\text{\rm L}(E,F)$ under the operator norm.
\end{lemma}

\begin{proof}
Suppose $\text{\rm a-G}(E,F)\ni T_k\stackrel{\|\cdot\|}{\to}T\in\text{\rm L}(E,F)$.
Let $(f_n)$ be disjoint $\text{\rm w}^\ast$-null in $F'$.
We need to show that $(T'f_n)$ is \text{\rm w}-null in $E'$. 
So, take any $g\in F''$. Let $\varepsilon>0$.
Pick a $k$ with $\|T'_k-T'\|\le\varepsilon$. Since
$T_k\in\text{\rm a-G}(E,F)$, there exists $n_0$ such that
$|g(T_k'f_n)|\le\varepsilon$ whenever $n\ge n_0$.
Note that $\|f_n\|\le M$ for some $M\in\mathbb{R}$
and for all $n\in\mathbb{N}$. Since $\varepsilon>0$
is arbitrary, it follows from
$$
   |g(T'f_n)|\le|g(T_k'f_n-T'f_n)|+|g(T'f_n)|\le
   \|g\|\|T_k'-T'\|\|f_n\|+\varepsilon\le(\|g\|M+1)\varepsilon
$$
for $n\ge n_0$, that $g(T'f_n)\to 0$. Since $g\in F''$ is arbitrary,
$T\in\text{\rm a-G}(E,F)$.

Similarly, the space $\text{\rm G}(E,F)$ is closed.
\end{proof}
\noindent
The next result follows from Theorem \ref{P-norm} and Lemma \ref{a-G are closed}.

\begin{corollary}\label{P-norm in a-G}
For arbitrary Banach lattices $E$ and $F$, 
$\text{\rm r-G}(E,F)$ and $\text{\rm r-a-G}(E,F)$ 
are both Banach spaces {\rm (}each under its enveloping norm{\rm )}.
\end{corollary}

We continue with the following definition.

\begin{definition}\label{def of DP operators} 
{\em A continuous operator $T:E\to F$ is called:
\begin{enumerate}[a)]
\item {\em Dunford--Pettis} (shortly, $T\in\text{\rm DP}(E,F)$) 
if $T$ takes \text{\rm w}-null sequences to 
norm null ones (cf. \cite[p.340]{AlBu}).
\item {\em weak Dun\-ford--Pet\-tis} (shortly, $T\in\text{\rm wDP}(E,F)$)
if $f_n(Tx_n)\to 0$ whenever $(f_n)$ is \text{\rm w}-null in $F'$ and 
$(x_n)$ is \text{\rm w}-null in $E$ \cite[p.349]{AlBu}.
\item 
{\em almost Dunford--Pettis} (shortly, $T\in\text{\rm a-DP}(E,F)$) if $T$ 
takes disjoint \text{\rm w}-null sequences to norm null ones \cite{San};
\item 
{\em almost weak Dun\-ford--Pet\-tis} (shortly, $T\in\text{\rm a-wDP}(E,F)$)
if $f_n(Tx_n) \to 0$ whenever $(f_n)$ is \text{\rm w}-null in $F'$ and 
$(x_n)$ is disjoint \text{\rm w}-null in $E$. 
\end{enumerate}}
\end{definition}
\noindent
Clearly, 
\begin{equation}\label{dpwdp1}
  \text{\rm K}(E,F)\subseteq\text{\rm DP}(E,F)
  \subseteq\text{\rm wDP}(E,F)\bigcap\text{\rm a-DP}(E,F), 
\end{equation}
and
\begin{equation}\label{dpwdp2}
  \text{\rm wDP}(E,F)\subseteq\text{\rm a-wDP}(E,F).  
\end{equation}
\noindent
The operator $L^1[0,1]\stackrel{T}{\to}\ell^\infty$, 
$Tf=\left(\int_0^1 f(t)r^+_k(t)\,dt \right)_{k=1}^\infty$ is \text{\rm wDP} 
yet not \text{\rm DP}, as $r_n\stackrel{\text{\rm w}}{\to}0$ and 
$\|Tr_n\|\ge \int_0^1 r_n(t)r^+_n(t)\,dt\equiv\frac{1}{2}$.
By \cite[Thm.4.1]{AqBo}, 
$\text{\rm a-DP}(E,F)=\text{\rm DP}(E,F)$ for all $F$ 
iff the lattice operations in $E$ are sequentially \text{\rm w}-continuous.
The identity operator: 
\begin{enumerate}[]
\item
$I:L^1[0,1]\to L^1[0,1]$ is \text{\rm a-DP} yet not \text{\rm DP};
\item
$I:c\to c$ is \text{\rm a-wDP} yet neither \text{\rm wDP} nor \text{\rm a-DP}.
\end{enumerate}
The domination property for \text{\rm a-DP}-operators 
was established in \cite[Cor.2.3]{AE}. We omit the proof of the 
next proposition as it is a modification of the proof 
of the Kalton--Saab domination theorem (cf. \cite[Thm.5.101]{AlBu}.

\begin{proposition}\label{dominated by awDP is awDP}
Any positive operator dominated by an \text{\rm a-wDP}-operator 
is likewise an \text{\rm a-wDP}-operator.
\end{proposition}
\noindent
We also omit the straightforward proof of the following fact.

\begin{lemma}\label{DP are closed}
The spaces $\text{\rm a-DP}(E,F)$ and $\text{\rm a-wDP}(E,F)$
are both closed in $\text{\rm L}(E,F)$ under the operator norm.
\end{lemma}
\noindent
By \cite[Cor.2.3]{AE}, Proposition \ref{dominated by awDP is awDP},
and Lemma \ref{DP are closed}, the next result follows from Theorem \ref{reg op banach lattice}. 

\begin{corollary}\label{P-norm in DP}
If $F$ is Dedekind complete then 
$\text{\rm r-a-DP}(E,F)$ and\\ $\text{\rm r-a-wDP}(E,F)$
are both Banach lattices under their enveloping norms.
\end{corollary}

\subsection{}
Finally, we discuss once more the Krengel example (cf. \cite[Ex.5.6]{AlBu}. 

\begin{example}\label{Krengel}
{\em Let $A_n$ be a $2^n\times 2^n$ matrix constructed inductively:
$$
A_1=\left[
\begin{matrix}
1 & 1 \\
1 & -1
\end{matrix}
\right] \quad \text{\rm and} \quad A_{n+1}=\left[
\begin{matrix}
A_n & A_n \\
A_n & -A_n
\end{matrix}
\right].
$$
For each $n$,  let $T_n:\ell^2_{2^n}\to\ell^2_{2^n}$ be an isometry
defined by the orthogonal matrix $2^{-\frac{n}{2}}A_n$. 
Then $|T_n|\in\text{\rm L}(\ell^2_{2^n}) $ is the operator, 
whose $2^n\times 2^n$ coordinate matrix has all entries equal to $2^{-\frac{n}{2}}$, 
in particular $\||T_n|\|=2^{\frac{n}{2}}$ for all $n$.
Consider the $c_0$-direct sum  
$E:=(\oplus_{n=1}^{\infty}\ell^2_{2^n})_0$. Note that $E$ 
has $\text{\rm o}$-continuous norm and hence is Dedekind complete. 
Define $T:E\to E$, by
\begin{equation}\label{1kr}
   Tx:=(2^{-\frac{n}{2}}T_n x_n) \ \ \ \ (x=(x_1,x_2,\dots,x_n,\dots)\in E).
\end{equation}
It is straightforward to see that $T\in\text{\rm K}(E)$, and then $T\in\text{\rm a-G}(E)$
by (\ref{limgr1}), and $T\in\text{\rm a-DP}(E)$ by (\ref{dpwdp1}).
The modulus $|T|$ exists and is given by
\begin{equation}\label{2kr}
   |T|x=(2^{-\frac{n}{2}}|T_n|x_n) \ \ \ \ (x\in E).
\end{equation}
However, $|T|\not\in\text{\rm a-G}(E)$. Indeed, consider
a disjoint $\text{\rm w}^\ast$-null sequence $(f_n)$ in 
$E'=(\oplus_{n=1}^{\infty}\ell^2_{2^n})_1$, where all 
terms in $f_n$ are zero, except the $n$-th term which 
is equal to $(2^{-\frac{n}{2}},2^{-\frac{n}{2}},\dots,2^{-\frac{n}{2}})\in\ell^2_{2^n}$.
Since $|T|'\in\text{\rm L}((\oplus_{n=1}^{\infty}\ell^2_{2^n})_1)$
satisfies $|T|'y=(2^{-\frac{n}{2}}|T_n|'y_n)$ for all $y\in E'$,
where the coordinate matrix of $|T_n|'$ has all entries equal 
to $2^{-\frac{n}{2}}$, then all terms of $|T|'f_n$ are zero, 
except the $n$-th term which is equal to 
$(2^{-\frac{n}{2}},2^{-\frac{n}{2}},\dots,2^{-\frac{n}{2}})\in\ell^2_{2^n}$.
Take $g\in E''=(\oplus_{k=1}^{\infty}\ell^2_{2^k})_\infty$ by letting each
$k$-th term of $g$ equals to 
$(2^{-\frac{k}{2}},2^{-\frac{k}{2}},\dots,2^{-\frac{k}{2}})\in\ell^2_{2^k}$, then 
$$
   g(|T|'f_n)=\sum\limits_{i=1}^{2^n}2^{-\frac{n}{2}}\cdot 2^{-\frac{n}{2}}=
   \sum\limits_{i=1}^{2^n}2^{-n}=1\not\to 0,
$$
and hence $T\not\in\text{\rm a-G}(E)$. In particular, 
$T\not\in\text{\rm a-Lm}(E)$ by (\ref{limgr1}). 
Now, take a disjoint $\text{\rm w}$-null sequence $(x_n)$ in 
$E=(\oplus_{n=1}^{\infty}\ell^2_{2^n})_0$, whose 
terms are zero, except the $n$-th one which 
is equals to $(2^{-\frac{n}{2}},2^{-\frac{n}{2}},\dots,2^{-\frac{n}{2}})\in\ell^2_{2^n}$.
Then $|T|x_n$ has all the terms zero, 
except the $n$-th term which is equal to 
$(2^{-\frac{n}{2}},2^{-\frac{n}{2}},\dots,2^{-\frac{n}{2}})$.
Thus $\||T|x_n\|=1$ for all $n$, and hence $|T|\not\in\text{\rm a-DP}(E)$.

Since $\text{\rm a-G}$- and $\text{\rm a-DP}$-operators satisfy the domination property,
it follows from Assertion \ref{prop elem}~(ii) that 
$T\in\text{\rm a-G}_r((\oplus_{n=1}^{\infty}\ell^2_{2^n})_0)\setminus
  \text{\rm r-a-G}((\oplus_{n=1}^{\infty}\ell^2_{2^n})_0)$ and\\ 
$T\in\text{\rm a-DP}_r((\oplus_{n=1}^{\infty}\ell^2_{2^n})_0)\setminus
  \text{\rm r-a-DP}((\oplus_{n=1}^{\infty}\ell^2_{2^n})_0)$.}
\end{example}

{\tiny 
}

\end{document}